\documentclass[11pt]{amsart}
\usepackage[top=4.4cm,bottom=4.4cm,left=3.4cm,right=3.4cm]{geometry}               
\usepackage{amsmath,amscd,amssymb,amsthm}
\usepackage[english]{babel}
\usepackage{mathtools}
\usepackage{enumerate}
\usepackage{setspace}
\usepackage{mathrsfs,makecell}
\usepackage{color,array}
\usepackage[numbers]{natbib}
\usepackage{subcaption}
\usepackage{caption}
\usepackage{tabularx}
\usepackage[hidelinks]{hyperref}

\DeclareMathOperator{\rk}{rk}

\DeclareMathOperator{\Hom}{Hom}
\DeclareMathOperator{\Gal}{Gal}
\DeclareMathOperator{\gal}{Gal}

\DeclareMathOperator{\Aut}{Aut}

\DeclareMathOperator{\Id}{Id}
\DeclareMathOperator{\GL}{GL}
\DeclareMathOperator{\SL}{SL}
\DeclareMathOperator{\Sp}{Sp}

\def\vF{\mathbb{F}}
\def\F{\mathbb{F}}

\def\Z{\mathbb{Z}}

\def\vR{\mathbb{R}}

\def\GL{\mathrm{GL}}
\def\SL{\mathrm{SL}}
\def\PSL{\mathrm{PSL}}
\def\Sp{\mathrm{Sp}}

\def\GG{G^{\text{geom}}}
\def\GA{G^{\text{arith}}}

\theoremstyle{definition}
\newtheorem{definition}{Definition}[section]
\newtheorem{remark}[definition]{Remark}

\theoremstyle{plain}
\newtheorem{theorem}[definition]{Theorem}
\newtheorem{corollary}[definition]{Corollary}
\newtheorem{lemma}[definition]{Lemma}
\newtheorem{proposition}[definition]{Proposition}

\newtheorem*{theorem*}{Theorem}

\author[A. Ferraguti]{Andrea Ferraguti}
\address{Instituto de Ciencias Matem\'aticas\\
Calle Nicol\'as Cabrera 13\\
28049 Madrid, Spain\\
}
\email{and.ferraguti@gmail.com}
\author[G. Micheli]{Giacomo Micheli}
\address{University of South Florida\\
4202 E Fowler Ave\\
33620 Tampa, US.
}
\email{gmicheli@usf.edu}
\title[Exceptional scatteredness in prime degree]{Exceptional scatteredness in prime degree}


\begin{document}

\begin{abstract}
Let $q$ be an odd prime power and $n$ be a positive integer. Let $\ell\in \F_{q^n}[x]$ be a $q$-linearised $t$-scattered polynomial of linearized degree $r$. Let $d=\max\{t,r\}$ be an odd prime number. In this paper we show that under these assumptions it follows that $\ell=x$. Our technique involves a Galois theoretical characterization of $t$-scattered polynomials combined with the classification of transitive subgroups of the general linear group over the finite field $\vF_q$.
\end{abstract}

\maketitle
\section{Introduction}
Scattered polynomials have risen a great deal of interest recently, due to their connections to Desarguesian spreads and maximal rank distance codes, which are used in the context of network coding. In fact, from a scattered polynomial one can construct scattered sets with respect to the Desarguesian spread (which are interesting objects in their own arising in finite geometry \cite{csajbok2017maximum,rottey2017geometric}) and in turn a maximal rank distance code for some parameters (see \cite{sheekey2016new,sheekey2019mrd}). 
Scatteredness is an extremely rare condition for a polynomial \cite[Section 1]{MR3812212}; this suggests the possibility of classifying all \emph{exceptional} $t$-scattered polynomials (see \cite{bartoli2019towards,MR3812212}). We recall below the definition of such objects. For a linearized polynomial $\sum_{i=0}^r a_ix^{q^i}$ with $a_r\neq 0$, the \emph{linearized degree} is the non-negative integer $r$.

\begin{definition}
Let $q$ be a prime power, $n,m$ be positive integers, $t$ be a non-negative integer and $\ell$ be a $q$-linearised polynomial in $\vF_{q^n}[x]$ of linearised degree $r$.
Then $\ell$ is said to be \emph{$(q,n,m,t)$-scattered} if, for any $s_0\in \vF_{q^{nm}}$, the polynomial $f-s_0x^{q^t}$ has at most $q$ roots. The non-negative integer $t$ is called \emph{index} of $\ell$. If, fixed $q,n$ and $t$, the polynomial $\ell$ is $(q,n,m,t)$-scattered for infinitely many $m$'s, we say that $\ell$ is \emph{exceptional $t$-scattered}.
\end{definition}

For $t=0$ exceptional $t$-scattered polynomials coincide with exceptional scattered polynomials, i.e.\ polynomials that are scattered for infinitely many extensions of the base field. On the other hand, if one finds a polynomial $\ell$ that is exceptional $t$-scattered, for $t>0$, then one automatically finds infinitely many scattered polynomials (over different fields) as follows:
for any $m$ such that $\ell$ is $t$-scattered over $\vF_{q^{nm}}$, consider the polynomial $\ell(x^{q^{nm-t}})$. It is elementary to see \cite[Section 1]{MR3812212} that $\ell(x^{q^{nm-t}})$ is scattered over $\vF_{q^{nm}}$.

In this paper we consider the problem of classifying exceptional $t$-scattered polynomials. This question has already been studied in the literature, leading to classification results in the case $t=0$ and $t=1$ \cite{bartoli2019towards,MR3812212}. Therefore, in the rest of the paper we simply restrict to the case $t\geq 1$, which simplifies some of the proofs without impacting on the generality of the results. Our work contains two main results: first we provide a Galois theoretical characterization of exceptional $t$-scattered polynomials (Theorem \ref{thm:chebotarev}); subsequently we use the latter to prove the following classification result.

\begin{theorem}\label{thm:main_thm}
Let $q$ be an odd prime power and let $t\geq 1$. Let $\ell$ be an exceptional $t$-normalized $t$-scattered polynomial of linearized degree $r$, and let $d\coloneqq \max\{r,t\}$. Suppose that $d$ is an odd prime. Then $\ell=x$.
\end{theorem}

The strategy of our proof is to exploit the full power of the classification of transitive subgroups of the general linear group \cite{dickson,hering} via our Galois theoretical characterization. It is worth noticing that our approach describes a new, completely equivalent condition to exceptional scatteredness.

In order to classify exceptional scattered polynomials, it is enough to consider those in a canonical form, which we are about to recall.

\begin{remark}\label{rmk:reductions}
As noticed in \cite[p.\ 511]{MR3812212}, if $\ell$ is $(q,n,m,t)$-scattered we can make the following assumptions without loss of generality:
\begin{itemize}
\item $\ell$ is monic;
\item the coefficient of the term $x^{q^t}$ of $\ell$ is zero;
\item if $t>0$ then the linear term of $\ell$ is non-zero.
\end{itemize}
In particular, notice that we can always assume that $t$ differs from the linear degree of $\ell$.
\end{remark}
\begin{definition}
A linearized polynomial satisfying the properties of Remark \ref{rmk:reductions} will be called $t$-\emph{normalized}.
\end{definition}

\section{A Galois theoretical characterization of $t$-normalized exceptional scattered polynomials}\label{sec:galois_characterization}

In this section we provide a Galois theoretical characterization of $t$-normalized exceptional scattered polynomials of positive index. We will use the notation and the terminology of \cite{stichtenoth}. We start by briefly recalling the setup and two auxiliary lemmata. Let $K$ be a function field with constant field $\vF_q$ and let $M$ be a Galois extension of $K$ with constant field $k$ (the field $k$ is a finite extension of $\F_q$). We denote by $\GA\coloneqq \Gal(M/K)$ the \emph{arithmetic Galois group} of the extension $M/K$. For every place $P$ of $K$ and every place $R$ of $M$ lying above $P$, the \emph{decomposition group} $D(R/P)\subseteq \GA$ is the set of all automorphisms fixing $R$ as a set. If $k_P$ and $k_R$ are the residue fields at $P$ and $R$, respectively, there is a surjective map $\pi_R\colon D(R|P)\twoheadrightarrow \Gal(k_R/k_P)$ whose kernel, denoted by $I(R|P)$, is called \emph{inertia group}. A \emph{Frobenius for $R|P$} is any preimage, via $\pi_R$, of the Frobenius automorphism $u\mapsto u^q$ of the extension of finite fields $k_R/k_P$. The \emph{geometric Galois group} of the extension $M/K$ is defined as $\GG\coloneqq \Gal(M/k\cdot K)$. This is a normal subgroup of $\GA$, and there is an isomorphism $\varphi\colon\GA/\GG\to \Gal(k/\F_q)$. Finally, for any $m\geq 1$ we denote by $M_m$ and $K_m$ the composita $M\cdot \F_{q^m}$ and $K\cdot \F_{q^m}$, respectively. The corresponding arithmetic and geometric Galois groups will be denoted by $\GA_m$ and $\GG_m$, the field of constants of $M_m$ will be denoted by $k_m$,  and $\varphi_m$ will be the isomorphism $\GA_m/\GG_m\to \Gal(k_m/\F_{q^m})$. The reader should notice that the isomorphism class of $\GG_m$ is independent of $m$, while $\GA_m$ might differ from $\GA$.

\begin{lemma}\label{lemma:chebotarev_consequence}
Let $M/K$ be a Galois extension of global function fields. Assume that the constant field of $K$ is $\vF_q$ and let $k$ be the constant field of $M$. Then there exists a constant $C\in \vR^+$, depending only on the degree of the extension $M/K$, with the following property: if $q^{m}>C$ then every $\gamma\in \GA_m$ such that $\varphi_m(\gamma)$ is the Frobenius automorphism for the extension $k_m/\F_{q^{m}}$ is also the Frobenius of a finite, unramified place of degree $1$ of $K_m$.
\end{lemma}
\begin{proof}
This is an immediate consequence of Chebotarev density theorem for function fields, see for example \cite[Remark 2.3]{permutation} or \cite{kosters2017short}.
\end{proof}

%The reader who is not be familiar with the fact stated in Lemma \ref{lemma:chebotarev_consequence} should think of it as follows.
%if we enlarge enough the field of constants of a function field extension $L/F$ (i.e. enlarging both the constant field of the base field field and of the extended field to a field $k$), a splitting condition determined by the Galois group of the closure of the extension, always appears in $kL/kF$ as long as the field of constants of $kF$ is the same as the field of constants of $F$. 
%A classical example of this remarkable fact is the following: if $f$ is a polynomial over a finite field, there exists a constant $D$ (depending only on the degree of the polynomial) such that $f$ is an exceptional permutation polynomial if and only if there exists a finite field $\vF_q$ such that $q>D$ and $f$ is permutation over $\vF_q$, see for example \cite[Chapter 8]{handbook}.

\begin{definition}
Let $M/K$ be a Galois extension of function fields. We call the smallest constant $C$ such that the conclusion of Lemma \ref{lemma:chebotarev_consequence} holds a \emph{uniformizing constant} for $M/K$.
\end{definition}

In other words, $C$ is the smallest size of a finite field $\F_{q^{m_0}}$ such that for every $m\geq m_0$, every element of $\GA_m$ that lies in the coset $\GG_m\gamma$ is the Frobenius of a finite, unramified place of $K$.

%and it has the property that $\varphi_m(\sigma)$ is the Frobenius automorphism of the field extension $k_m/\F_{q^{nm}}$. We refer the reader to \cite{stichtenoth} for the definitions of the ramification index and the relative degree.

The following lemma is a classical fact from algebraic number theory, whose proof can be found for example in \cite{guralnick2007exceptional}.

\begin{lemma}\label{orbits}
Let $L/K$ be a finite separable extension of global function fields, let $M$ be its Galois closure and $G\coloneqq \Gal(M/K)$ be its (arithmetic) Galois group. Let $P$ be a place of $K$ and $\mathcal Q$ be the set of places of $L$ lying above $P$.
Let $R$ be a place of $M$ lying above $P$. The following hold:
\begin{enumerate}
\item There is a natural bijection between $\mathcal Q$ and the set of orbits of $H\coloneqq\Hom_K(L,M)$ under the action of the decomposition group $D(R|P)=\{g\in G\,|\, g(R)=R\}$.
\item  Let $Q\in \mathcal Q$ and let $H_Q$ be the orbit of $D(R|P)$ corresponding to $Q$. Then $|H_Q|=e(Q|P)f(Q|P)$ where $e(Q|P)$ and $f(Q|P)$ are ramification index and relative degree, respectively. 

\item The orbit $H_Q$ partitions further under the action of the inertia group $I(R|P)$ into $f(Q|P)$ orbits of size $e(Q|P)$. 

\end{enumerate}
\end{lemma}

The above lemmata allows to convert splitting conditions of intermediate extension into group theoretical ones (see \cite{micheli2019constructions,micheli2019selection} for more applications of these lemmata not related to exceptionality ).

\vspace{3mm}

\subsection*{A function field theoretical setup for scattered polynomials.}
Let $\ell\in \F_{q^n}[x]$ be a $q$-linearized polynomial of linearized degree $r\geq 0$ and such that the linear term of $\ell$ is non-zero. Let $s$ be transcendental over $\F_q$ and let $t>0$. We will call $d\coloneqq \max\{r,t\}$ the \emph{$t$-scatter-degree} of $f$. Let $M$ be the splitting field of $\ell-sx^{q^t}$ over $\F_{q^n}(s)$.  For every $m\geq 1$, we denote by $M_m$ the compositum $M\cdot \F_{q^{nm}}$ and by $\GA_m\coloneqq \gal(M_m/\F_{q^{nm}}(s))$ the arithmetic Galois group of the extension $M_m/\F_{q^{nm}}(s)$. The field of constants of such extension, which is defined as $\overline{\F}_q\cap M_m$, will be denoted by $k_m$. We denote by $\GG_m\coloneqq \gal(M_m/k_m\cdot \F_q(s))$ the geometric Galois group. We will denote by $\varphi_m$ the isomorphism $\GA_m/\GG_m\to \Gal(k_m/\F_{q^{nm}})$. Finally, we will denote by $V$ the set of roots of $\ell-sx^{q^t}$ in an algebraic closure $\overline{\F_q(s)}$ of $\F_q(s)$.

One can immediately deduce the following corollary from Lemma \ref{orbits}.
\begin{corollary}\label{cor:orbits}
Let $t\geq 1$ and let $\ell$ be a linearized polynomial in $t$-normalized form.
Let $\alpha$ be a root of $\ell/x-sx^{q^t-1}$ and let $L\coloneqq\F_{q^{nm}}(\alpha)$. Let $P$ be a place of $\F_{q^{nm}}(s)$ and $\mathcal Q$ be the set of places of $L$ lying above $P$.
Let $R$ be a place of $M_m$ lying above $P$. Then the following hold:
\begin{enumerate}
\item There is a natural bijection between $\mathcal Q$ and the set of orbits of $V$ under the action of the decomposition group $D(R|P)$.
\item  Let $Q\in \mathcal Q$ and let $V_Q$ be the orbit of $D(R|P)$ corresponding to $Q$. Then $|V_Q|=e(Q|P)f(Q|P)$ where $e(Q|P)$ and $f(Q|P)$ are the ramification index and the relative degree, respectively. 
\end{enumerate}
\end{corollary}
\begin{proof}
Using Lemma \ref{orbits}, it is enough to observe that since $\ell$ is $t$-normalized, then the polynomial $\ell/x-sx^{q^t-1}$ is separable and it is irreducible because $\gcd(\ell/x,x^{q^t-1})=1$. Hence, $\Hom_K(L,M)$ and the set $V\setminus \{0\}$ of roots of $\ell/x-sx^{q^t-1}$ are isomorphic $D(R|P)$-sets.
\end{proof}

With this notation, we are now ready to state and prove the Galois theoretical characterization of exceptional scattered polynomials, which will allow to encode scatteredness in group theoretical terms.

\begin{remark}
Notice that the set $V$ of roots of $\ell-sx^{q^t}$ is an $\F_q$-vector space. Since $\GA$ and $\GG$ act $\F_q$-linearly on $V$, it follows that they are both subgroups of $\Aut_{\F_q}(V)\cong\GL_d(\F_q)$.
\end{remark}

\begin{theorem}\label{thm:chebotarev}
Let $\ell\in\vF_{q^n}$ be a $q$-linearized polynomial of linearized degree $r$ and in $t$-normalized form. Let $t\geq 1$ be a positive integer. Let $d$ be the $t$-scatter-degree of $\ell$. Let $C$ be a uniformizing constant for the extension $M/\vF_{q^n}(s)$, and let $m\geq 1$ be such that $q^{mn}>C$. Then the following are equivalent:
\begin{enumerate}
\item $\ell$ is $(q,n,m,t)$-scattered;
\item for every $\gamma\in \GA_m$ such that $\varphi_m(\gamma)$ is a Frobenius for $k_m/\F_{q^{mn}}$ and every $h\in \GG_m$, the following condition holds:
$$\rk (h\gamma-\Id)\geq d-1.$$
\end{enumerate}
Moreover, $\ell$ is exceptional $t$-scattered if and only if there exists an $m\geq 1$ such that $q^{nm}>C$ and $\ell$ is $(q,n,m,t)$-scattered.
\end{theorem}
\begin{proof}
First, notice that $\ell(x)/x-sx^{q^t-1}$ is an irreducible polynomial in $\F_{q^{nm}}(s)[x]$. Let us set $L_m\coloneqq \F_{q^{nm}}(s)[x]/(\ell(x)/x-sx^{q^t-1})$, which is simply $\F_{q^{nm}}(s,\alpha)$, where $\alpha$ is a non-zero root of $ \ell-sx^{q^t}$.

(1)$\implies$ (2) Let $\ell$ be $(q,n,m,t)$-scattered, and pick $\gamma\in \GA_m$ with $\varphi_m(\gamma)$ the Frobenius automorphism for $k_m/\F_{q^{mn}}$ and any $h\in \GG_m$. By Lemma \ref{lemma:chebotarev_consequence}, $h\gamma$ is a Frobenius for a finite, unramified place of degree 1 of $\F_{q^{mn}}(s)$ which we will denote by $P$. The polynomial $\ell/x-s_0x^{q^t-1}$, where $s_0$ is the value in $\vF_q$ corresponding to $P$, has at most $q-1$ roots in $\F_{q^{mn}}$. Hence there are at most $q-1$ finite places of degree 1 of $L_m$ lying above $P$. 
Let $V$ be the $\F_q$-vector space of roots of $\ell-sx^{q^t}$, so that $V\setminus \{0\}$ is the set of roots of $\ell/x-s_0x^{q^t-1}$.
By Corollary \ref{cor:orbits} it follows that the decomposition group $D(R|P)\subseteq \GA_m$ has at most $q-1$ fixed points when acting on $V\setminus \{0\}$. Since $R|P$ is unramified, $D(R|P)$ is generated by $h\gamma$ and therefore $h\gamma$ has at most $q-1$ fixed points when acting on $V\setminus \{0\}$. But this is equivalent to asking that $h\gamma-\Id$ has rank at least $d-1$.

(2)$\implies$ (1) Suppose by contradiction that $\ell$ is not $(q,n,m,t)$-scattered. Then there exists $s_0\in \F_{q^{nm}}$ such that $\ell-s_0x^{q^t}$ has at least $q^2-1$ non-zero roots in $\F_{q^{nm}}$ (notice in fact that if $\ell/x-s_0x^{q^t-1}$ has at least $q$ roots, then it has at least $q^2-1$ roots). Then there are at least $q^2-1$ finite places of degree 1 of $L_m$ lying above the place $P_{s_0}$ of $\F_{q^{nm}}(s)$ corresponding to $s_0$. Notice that these places are unramified in $L_m$: if $P_{s_0}$ ramifies in $L_m$ then $\ell-s_0x^{q^t}$ has a multiple root. But this is impossible by assumption, because $t>0$ and $\ell$ has a non-zero linear term, and therefore the derivative of $\ell-s_0x^{q^t}$ is constant.

Hence all finite places of $L_m$ lying above $P_{s_0}$ are unramified, and at least $q^2-1$ of them have degree 1. Fix a place $R$ of $M_m$ lying above $P_{s_0}$.  By Lemma \ref{cor:orbits}, the decomposition group $D(R|P_{s_0})$ has at least $q^2-1$ fixed points when acting on $V\setminus \{0\}$. Thus all elements of $D(R|P_{s_0})$ enjoy the same property. Now notice that there is a surjective map $D(R|P_{s_0})\twoheadrightarrow \Gal(k_R/k_{P_{s_0}})$. Since $P_{s_0}$ has degree 1 then the residue field $k_{P_{s_0}}$ is $\F_{q^{mn}}$, and it is a standard fact that $k_R$ is an extension of $k_m$. Hence, there must exist $\gamma\in D(R|P_{s_0})$ whose image via the aforementioned surjection is the Frobenius for $k_m/\F_{q^{mn}}$. It follows immediately that also $\varphi_m(\gamma)$ is the Frobenius for $k_m/\F_{q^{mn}}$. But then we have a contradiction with (2) because if $\gamma$ acts on $V\setminus \{0\}$ with at least $q^2-1$ fixed points, then $\rk(\gamma-\Id)\leq d-2$.

Finally, let us prove the last part of the statement. One direction is obvious. Thus, suppose $m\geq 1$ is such that $q^{nm}>C$ and $\ell$ is $(q,n,m,t)$-scattered. Then condition (2) holds. Now for all the infinitely many $y\in \Z_{\geq 1}$ that are coprime with $[k_m:\F_{q^{nm}}]$ we have that $\GA_m\cong\GA_{my}$ and $\GG_m\cong\GG_{my}$, and therefore (2) holds for infinitely many $y$'s. It follows that $\ell$ is exceptional $t$-scattered. 
\end{proof}

\begin{corollary}\label{cor:chebotarev}
With the notation of Theorem \ref{thm:chebotarev}, if $\ell$ is $(q,n,m,t)$-scattered then $\GG_m\neq \GA_m$.
\end{corollary}
\begin{proof}
If it was $\GG_m=\GA_m$, then $\varphi_m(\Id)$ would be a Frobenius for $k_m/\F_{q^{nm}}$, and hence one should have, in particular, that $\rk(\Id-\Id)\geq d-1$, a clear contradiction.
\end{proof}

\section{Finite transitive linear groups in odd characteristic}

In \cite{hering}, Hering classified all finite 2-transitive affine groups with abelian elementary socle (over finite fields). Looking at the stabilizer of a point, one can deduce the following classification theorem for finite transitive linear groups. Recall that if $q$ is a prime power and $a,b\geq 1$ we denote by $\Gamma L_a(\F_{q^b})$ the general semilinear group, namely the semidirect product $\GL_a(\F_{q^b})\rtimes \Aut_{\vF_q}(\F_{q^b})\cong \GL_a(\F_{q^b})\rtimes C_b$.

\begin{theorem}[{{\cite[Theorem 69.7]{handbook}}}]\label{class}
Let $p$ be an odd prime, $n\in \Z_{\geq 1}$ and $G$ be a subgroup of $\GL_n(\F_p)$ that acts transitively on $\F_p^n\setminus \{0\}$. Then $G$ satisfies one of the following.
\begin{enumerate}
\item $\SL_{e}(\F_{p^{n/e}})\leq G\leq \Gamma L_{e}(\F_{p^{n/e}})$ for some $e\mid n$;
\item $\Sp_{e}(\F_{p^{n/e}})\leq G\leq \Gamma L_{e}(\F_{p^{n/e}})$ for some even $e\mid n$;
\item $n$ is even and $(p,n)$ belongs to a finite list of sporadic pairs.
\end{enumerate}
\end{theorem}

When $d$ is prime, one can in turn deduce the following proposition.

\begin{proposition}\label{thm:classification}
Let $q$ be an odd prime power, $d$ be an odd prime, and $G$ be a subgroup of $\GL_d(\F_q)$ that acts transitively on $\F_q^d\setminus \{0\}$. Then $G$ satisfies one of the following.
\begin{enumerate}
\item $\SL_{d}(\F_q)\leq G\leq \GL_d(\F_q)$;
\item $G\leq \Gamma L_1(\F_{q^d})$.
\end{enumerate}
\end{proposition}
\begin{proof}
Let $q=p^a$ where $p$ is a prime and $a\geq 1$. Since $\GL_d(\F_q)\leq \GL_{ad}(\F_p)$ and the group $G$ acts transitively on $\F_p^{ad}\setminus \{0\}$, we can apply Theorem \ref{class} with $n=ad$ and deduce that one of the following holds:
\begin{enumerate}
\item[(I)] $\SL_{e}(\F_{p^{ad/e}})\leq G\leq \Gamma L_{e}(\F_{p^{ad/e}})$ for some $e\mid ad$;
\item[(II)] $\Sp_{e}(\F_{p^{ad/e}})\leq G\leq \Gamma L_{e}(\F_{p^{ad/e}})$ for some even $e\mid ad$;
\item[(III)] $ad$ is even and $(p,n)$ belongs to a finite list of sporadic pairs.
\end{enumerate}
First, suppose that (I) holds for some $e\geq 3$. We claim that $e=d$, which implies immediately that $\SL_{d}(\F_q)\leq G\leq \GL_d(\F_q)$, since $\Gamma L_d(\F_q)=\GL_d(\F_q)$. Let us now split the proof into two cases.

\textbf{Case 1: $d\mid e$}. Since $\SL_{e}(\F_{p^{ad/e}})\leq G$ and $G\leq \GL_d(\F_q)$ by hypothesis, we must have that $v_p(|\SL_{e}(\F_{p^{ad/e}})|)\leq v_p(|\GL_d(\F_q)|)$, where $v_p$ is the usual $p$-adic valuation. However, a quick calculation shows that $v_p(|\SL_{e}(\F_{p^{ad/e}})|)=ad(e-1)/2$, while $v_p(|\GL_d(\F_q)|)=ad(d-1)/2$, forcing $e=d$. 

\textbf{Case 2: $d\nmid e$}. Also in this case we have $\SL_{e}(\F_{p^{ad/e}})\leq G$ and $G\leq \GL_d(\F_q)$ by hypothesis. The condition $d\nmid e$ and the fact that $d$ is prime imply that $e\mid a$ and $(d,e)=1$. Let $\widetilde{p}\coloneqq p^{a/e}$. Let $r$ be a Zsigmondy prime for $\widetilde{p}^{d(e-1)}-1$, i.e.\ a prime that divides $\widetilde{p}^{d(e-1)}-1$ but does not divide $\widetilde{p}^t-1$ for any $1\leq t<d(e-1)$, whose existence is guaranteed by Zsigmondy's Theorem (see \cite[Theorem V]{birkhoff}). It is immediate to check that $r\mid |\SL_e(\F_{p^{ad/e}})|$ but $r\nmid |\GL_d(\F_q)|$, as $et$ can't be a multiple of $d(e-1)$ for any $t\leq d$, yielding a contradiction.

If (II) holds for some $e\geq 6$, one argues in a totally analogous way. For $e=4$, just notice that $|\Sp_4(\F_{p^{ad/4}})|=p^{ad}(p^{ad/2}-1)(p^{ad}-1)$. Looking at a Zsigmondy prime or $p^{ad/2}-1$, one sees immediately that $|\Sp_4(\F_{p^{ad/4}})|$ cannot divide $|\GL_d(\F_q)|$.

Next, assume (I) holds for $e=2$, which is the same as saying that (II) holds for $e=2$ (as $\SL_2(\vF_Q)=\Sp_2(\vF_Q)$ for any prime power $Q$). Since $d$ is odd, $a$ must be even. Let $\widetilde{p}\coloneqq p^{a/2}$. We claim that there is no embedding $\SL_2(\F_{\widetilde{p}^d})\hookrightarrow \GL_d(\F_{\widetilde{p}^2})$. Suppose by contradiction that there is one. Since both $p$ and $d$ are odd primes, $\SL_2(\F_{\widetilde{p}^d})$ is perfect and $[\GL_d(\F_{\widetilde{p}^2}),\GL_d(\F_{\widetilde{p}^2})]=\SL_d(\F_{\widetilde{p}^2})$. Hence, if such embedding exists then there exists also an embedding $\iota\colon\SL_2(\F_{\widetilde{p}^d})\hookrightarrow \SL_d(\F_{\widetilde{p}^2})$. The group $\SL_2(\F_{\widetilde{p}^d})$ contains an element $\sigma$ of order $\widetilde{p}^d+1$, which comes from a Singer cycle of $\GL_2(\F_{\widetilde{p}^d})$. Now let $r$ be a Zsigmondy prime for $\widetilde{p}^{2d}-1$. Clearly $r\mid \widetilde{p}^d+1$. Hence, the cyclic subgroup generated by $\iota(\sigma)$ contains an element $\tau$ of order $r$. Since $r$ is a Zsigmondy prime for $\widetilde{p}^{2d}-1$, one can check that the cyclic group $H$ generated by $\tau$ acts irreducibly on $\F_{\widetilde{p}^2}^d$. Hence, by Schur's Lemma and the fact that finite division rings are fields, the centralizer $C(H)$ of $H$ in $\GL_d(\F_{\widetilde{p}^2})$ is the multiplicative group of a finite field of characteristic $p$. Since $r\mid |C(H)|$, then it must be $|C(H)|=\widetilde{p}^{2d}-1$; in other words $C(H)$ arises from a Singer cycle. Thus its intersection with $\SL_d(\F_{\widetilde{p}^2})$ has order $(\widetilde{p}^{2d}-1)/(\widetilde{p}^2-1)$, which is not divisible by $\widetilde{p}^d+1$. This gives a contradiction, so there is no such embedding and (I) cannot hold for $e=2$.

Finally, the only transitive sporadic group that can appear as a subgroup of $\GL_{ad}(\F_p)$ with $p,d$ odd primes is $\SL_2(\F_{13})$, which is contained in $\GL_6(\F_3)$. This necessarily comes from the setting in which $q=9$ and $d=3$. But one can check with Magma \cite{magma} that $\SL_2(\F_{13})$ does not embed in $\GL_3(\F_9)$, concluding the proof of the proposition.
\end{proof}

\section{Exceptional scattered polynomials of prime degree}
The goal of this section is to prove Theorem \ref{thm:main_thm}. The proof relies on two fundamental ingredients: the Galois theoretical characterization of exceptional scattered polynomials given by Theorem \ref{thm:chebotarev} and the classification of finite transitive groups given by Proposition \ref{thm:classification}.

Let us recall the setup, which is the one explained in Section \ref{sec:galois_characterization}. Let $\ell\in \F_{q^n}[x]$ be a $q$-linearized, exceptional $t$-normalized $t$-scattered polynomial of linearized degree $r$, and let $t>0$. Let $M$ be the splitting field of $\ell-sx^{q^t}$ over $\F_{q^n}(s)$; for every $m\geq 1$ let $M\coloneqq M\cdot \F_{q^{nm}}(s)$, $k_m\coloneqq \overline{\F}_q\cap M_m$ and let $\GA_m\coloneqq \Gal(M_m/\F_{q^{nm}}(s))$ and $\GG_m\coloneqq \Gal(M_m/k_m\cdot \F_{q^{n}}(s))$. Finally, we let $d\coloneqq \max\{r,t\}$ and $V$ be the set of roots of $\ell-sx^{q^t}$, which is an $\F_q$-vector space.

We begin by showing that in the setting of Theorem \ref{thm:main_thm}, there is a very restrictive condition on $\GA_m$ and $\GG_m$.
\begin{lemma}\label{lemma:monodromy}
Let $t\geq 2$, $n$ be an integer, $q$ be a prime power and let $\ell\in \F_{q^n}[x]$ be a $q$-linearized, exceptional $t$-normalized $t$-scattered polynomial of linearized degree $r$. Let $M,\GA,\GG$ be defined as above. Assume that $d\coloneqq\max\{r,t\}$ is an odd prime. Let $C$ be the uniformizing constant for the extension $M/\vF_{q^{n}}(s)$, and suppose $m\geq 1$ is such that $q^{nm}>C$ and $\ell$ is $(q,n,m,t)$-scattered. Then $|\GG_m|=q^d-1$ and $\GA_m\cong\Gamma L_1(q^d)$.
\end{lemma}
\begin{proof}
Before starting the proof, let us recall the following elementary facts, that will be helpful later on:
\begin{equation}\label{eq:useful}
|\GL_d(\F_q)|=\prod_{i=0}^{d-1}(q^d-q^i),\qquad \GL_d(\F_q)\cong \SL_d(\F_q)\rtimes \F_q^*.
\end{equation}
so that in particular $(q-1)|\SL_d(\F_q)|=|\GL_d(\F_q)|$.

Now, since $(\ell-sx^{q^t})/x$ is absolutely irreducible, $\GG_m$ acts transitively on $V\setminus\{0\}$. Thus by Proposition \ref{thm:classification} we only have two possibilities for $\GG_m$: either $\SL_d(\F_q)\leq \GG_m\leq \GL_d(\F_q)$ or $\GG_m\leq \Gamma L_1(q^d)$. As a first step of the proof, we will show that $\SL_d(\F_q)\leq \GG_m\leq \GL_d(\F_q)$ leads to a contradiction. In fact, suppose that such inclusion holds, and let $\gamma\in \GA$ be a Frobenius for the field of constants $k_m$ of $M_m$. Any element of the coset $\GG_m\gamma$ is a Frobenius for $k_m$. Hence, up to multiplying on the left by an element of $\SL_d(\F_q)$, we can assume by \eqref{eq:useful} that $\gamma$ has the form $(\lambda_{i,j})_{i,j=1,\ldots,d}$ where $\lambda_{1,1}=\lambda\in \F_q^*$, $\lambda_{i,i}=1$ for $i>1$ and $\lambda_{i,j}=0$ otherwise. This can be seen by noticing that a splitting of the short exact sequence $\SL_d(\F_q)\hookrightarrow \GL_d(\F_q)\twoheadrightarrow \F_q^*$ is the map $\F_q^*\hookrightarrow \GL_d(\F_q)$ that sends each $\lambda\in \F_q^*$ to the matrix $(\lambda_{i,j})_{i,j}$ described above. But then, $\rk(\gamma-\Id)\leq 1<d-1$, contradicting Theorem \ref{thm:chebotarev}.

Hence, we must have $\GG_m\leq \Gamma L_1(q^d)$. Since $\GG_m$ acts transitively on $V\setminus\{0\}$, which has cardinality $q^d-1$, it follows by the Orbit-Stabilizer Theorem that $(q^d-1)\mid |\GG_m|$. On the other hand, $|\Gamma L(1, q^d )|=d(q^d-1)$ and hence $|\GG_m|\in \{q^d-1,d(q^d-1)\}$ because $d$ is prime. Now we claim that we must have $\GA_m\leq \Gamma L(1,q^d)$. Notice that this ends the proof, because $\GG_m\trianglelefteq \GA_m$, and hence it can only be $\GA_m=\GG_m$ or $\GA_m=\Gamma L(1,q^d)$; however $\GA_m=\GG_m$ cannot occur by Corollary \ref{cor:chebotarev}.

To prove that $\GA_m\leq \Gamma L(1,q^d)$, we proceed by contradiction. By Proposition \ref{thm:classification}, the only other possibility is that $\SL_d(\F_q)\trianglelefteq \GA_m\leq \GL_d(\F_q)$. If this holds, since $\GG_m\trianglelefteq \GA_m$ we must have that $\GG_m\cap \SL_d(\F_q)\trianglelefteq \SL_d(\F_q)$. If $\GG_m\cap \SL_d(\F_q)=\SL_d(\F_q)$ then $|\SL_d(\F_q)|$ must divide either $q^d-1$ or $d(q^d-1)$, and it is immediate to see that this is impossible because of \eqref{eq:useful}. On the other hand, since $d\geq 3$ and $\PSL_d(\vF_q)$ is simple, the only non-trivial normal subgroups of $\SL_d(\F_q)$ are those contained in the intersection of $\SL_d(\F_q)$ with the center of $\GL_d(\F_q)$ (see \cite{dickson}). Thus if $\GG_m\cap \SL_d(\F_q)\neq\SL_d(\F_q)$ then we have that, in particular, $|\SL_d(\F_q)\cap \GG_m|\leq q-1$. Hence
$$|\GL_d(\F_q)|\geq |\GG_m\cdot \SL_d(\F_q)|=\frac{|\GG_m||\SL_d(\F_q)|}{|\GG_m\cap \SL_d(\F_q)|}\geq \frac{(q^d-1)|\GL_d(\F_q)|}{(q-1)^2},$$
which is a contradiction because $d\geq 3$ and therefore $\displaystyle |\GL_d(\F_q)|<\frac{(q^d-1)|\GL_d(\F_q)|}{(q-1)^2}$.
\end{proof}

\begin{proof}[Proof of Theorem \ref{thm:main_thm}]
Let $C$ be the uniformizing constant for $M/\vF_{q^n}(s)$. Since $\ell$ is exceptional $t$-scattered, there exists $m\geq 1$ such that $q^{mn}>C$ and $\ell$ is $(q,n,m,t)$-scattered. By Lemma \ref{lemma:monodromy} it follows immediately that the constant field of $M_m$ is $k_m=\F_{q^{dnm}}$ and that the splitting field $M_m$ of $\ell-sx^{q^t}$ over $\F_{q^{dnm}}(s)$ is simply $\F_{q^{dnm}}(s)[x]/(\ell/x-sx^{q^t-1})$ because it contains one root and it has the correct degree. Notice that $M_m$ is isomorphic to the rational function field over $\F_{q^{dnm}}$, and hence its places of degree 1 are in 1-1 correspondence with the $\F_{q^{dnm}}$-rational points of the projective line.

From now on, we will denote by $P_0$ and $P_\infty$ the places of $\F_{q^{dnm}}(s)$ corresponding to $0$ and to $\infty$, respectively. Similarly, we will denote by $Q_0$ and $Q_\infty$ the places of $M_m$ corresponding to $0$ and to $\infty$, respectively.

%First, suppose that $t=0$. Since $\ell$ is normalized, it does not have the linear term, and therefore we can write $\ell=\sum_{i=j}^da_ix^{q^i}$ for some $j>0$ with $a_j\neq 0$ and $a_d=1$. To find the decomposition pattern of $P_0$ in $M$ one has to look at the factorization of $\ell/x$, which is: $x^{q^j-1}(\sum_{i=j}^da_ix^{q^{i-j}-1})^{q^j}$. Hence, since $a_j>0$, there is at least a place of $M$ lying over $P_0$ with ramification index $q^j$, which is impossible because $M/\F_{q^d}(s)$ is a Galois extension of order prime to $q$. It follows that $j=0$, i.e.\ $\ell$ is the monomial $x^{q^d}$.

Now we have to distinguish two cases, namely $t<r$ and $t>r$.

If $t<r$, then consider the place $P_\infty$. The only two places of $M_m$ that lie above $P_\infty$, are $Q_0$ and $Q_\infty$. It is immediate to check that the ramification index of $Q_0/P_\infty$ is $(q^r-1)-(q^t-1)=q^t(q^{r-t}-1)$. Since $t>0$, it follows that this ramification index is divisible by $q$, and this is impossible since $M_m/\F_{q^{dnm}}(s)$ is a Galois extension of degree $q^d-1$, which is coprime with $q$.

If $t>r$, we look at $P_0$: the place $Q_\infty$ lies above it and has ramification index $q^t-q^r$. If $r>0$, this is divisible by $q$ and we are led to a contradiction by the same argument we used in the case $t<r$. Therefore it must be $r=0$, i.e.\ $\ell=x$.
\end{proof}

\bibliographystyle{plain}
\bibliography{bibliography}

\begin{thebibliography}{10}

\bibitem{bartoli2019towards}
Daniele Bartoli and Maria Montanucci.
\newblock Towards the full classification of exceptional scattered polynomials.
\newblock \url{https://arxiv.org/abs/1905.11390}, 2019.

\bibitem{MR3812212}
Daniele Bartoli and Yue Zhou.
\newblock Exceptional scattered polynomials.
\newblock {\em J. Algebra}, 509:507--534, 2018.

\bibitem{birkhoff}
Geo.~D. Birkhoff and H.~S. Vandiver.
\newblock On the integral divisors of {$a^n-b^n$}.
\newblock {\em Ann. of Math. (2)}, 5(4):173--180, 1904.

\bibitem{magma}
Wieb Bosma, John Cannon, and Catherine Playoust.
\newblock The {M}agma algebra system. {I}. {T}he user language.
\newblock {\em J. Symbolic Comput.}, 24(3-4):235--265, 1997.
\newblock Computational algebra and number theory (London, 1993).

\bibitem{csajbok2017maximum}
Bence Csajb\'{o}k, Giuseppe Marino, Olga Polverino, and Ferdinando Zullo.
\newblock Maximum scattered linear sets and {MRD}-codes.
\newblock {\em J. Algebraic Combin.}, 46(3-4):517--531, 2017.

\bibitem{dickson}
Leonard~Eugene Dickson.
\newblock Theory of linear groups in an arbitrary field.
\newblock {\em Trans. Amer. Math. Soc.}, 2(4):363--394, 1901.

\bibitem{permutation}
Andrea Ferraguti and Giacomo Micheli.
\newblock Full classification of permutation rational functions and complete
  rational functions of degree three over finite fields.
\newblock {\em Des. Codes Cryptogr. (to appear)}, 2020.

\bibitem{guralnick2007exceptional}
Robert~M. Guralnick, Thomas~J. Tucker, and Michael~E. Zieve.
\newblock Exceptional covers and bijections on rational points.
\newblock {\em Int. Math. Res. Not. IMRN}, (1):Art. ID rnm004, 20, 2007.

\bibitem{hering}
Christoph Hering.
\newblock Transitive linear groups and linear groups which contain irreducible
  subgroups of prime order. {II}.
\newblock {\em J. Algebra}, 93(1):151--164, 1985.

\bibitem{handbook}
Norman~L. Johnson, Vikram Jha, and Mauro Biliotti.
\newblock {\em Handbook of finite translation planes}, volume 289 of {\em Pure
  and Applied Mathematics (Boca Raton)}.
\newblock Chapman \& Hall/CRC, Boca Raton, FL, 2007.

\bibitem{kosters2017short}
Michiel Kosters.
\newblock A short proof of the {C}hebotarev density theorem for function
  fields.
\newblock {\em Math. Commun.}, 22(2):227--233, 2017.

\bibitem{micheli2019constructions}
Giacomo Micheli.
\newblock Constructions of locally recoverable codes which are optimal.
\newblock {\em IEEE Transactions on Information Theory}, 2019.

\bibitem{micheli2019selection}
Giacomo Micheli.
\newblock On the selection of polynomials for the {DLP} quasi-polynomial time
  algorithm for finite fields of small characteristic.
\newblock {\em SIAM J. Appl. Algebra Geom.}, 3(2):256--265, 2019.

\bibitem{rottey2017geometric}
Sara Rottey and John Sheekey.
\newblock A geometric characterisation of {D}esarguesian spreads.
\newblock {\em J. Algebraic Combin.}, 46(2):455--474, 2017.

\bibitem{sheekey2016new}
John Sheekey.
\newblock A new family of linear maximum rank distance codes.
\newblock {\em Adv. Math. Commun.}, 10(3):475--488, 2016.

\bibitem{sheekey2019mrd}
John Sheekey.
\newblock {MRD} codes: Constructions and connections.
\newblock \url{https://arxiv.org/abs/1904.05813}, 2019.

\bibitem{stichtenoth}
Henning Stichtenoth.
\newblock {\em Algebraic function fields and codes}, volume 254 of {\em
  Graduate Texts in Mathematics}.
\newblock Springer-Verlag, Berlin, second edition, 2009.

\end{thebibliography}

\end{document}